\documentclass[12pt]{article}
\usepackage{graphicx,amssymb,amsmath,amsfonts,amscd}
\usepackage[pdftex]{color}
\usepackage{amsfonts}
\usepackage{fullpage, amsmath, amsthm,amsfonts,amssymb,stmaryrd, mathrsfs}
\usepackage{hyperref}
\usepackage[pdftex]{color}
\usepackage[all]{xy}

\newtheorem{thm}{Theorem}[section]

\newtheorem{cor}[thm]{Corollary}

\newtheorem{lem}[thm]{Lemma}

\newtheorem{rem}[thm]{Remark}

\newtheorem{question}[thm]{Question}

\newtheorem{lemma}[thm]{Lemma}

\newcommand{\Sym}{\mathrm{Sym}}

\newcommand{\Fq}{\mathbb F_q}

\newcommand{\Z}{\mathbb Z}

\newcommand{\Gal}{\mathrm {Gal}}

\begin{document}

\title{Artin Conjecture for $p$-adic Galois Representations of Function Fields}

\author{Ruochuan Liu\\Beijing International Center for Mathematical Research \\Peking University, Beijing, 100871\\liuruochuan@math.pku.edu.cn\\
\\
Daqing Wan\\Department of Mathematics\\University of California, Irvine, CA 92697-3875\\dwan@math.uci.edu}


\maketitle

\abstract{For a global function field $K$ of positive characteristic $p$, we show that Artin's entireness conjecture  
for L-functions of geometric $p$-adic Galois representations of $K$ is true in a non-trivial $p$-adic disk but is false in the full $p$-adic plane.}
In particular, we prove the non-rationality\footnote{Acknowledgements. It is a pleasure to thank Jean-Yves Etesse whose persistent interest in the non-rationality 
of the geometric unit root L-function motivated us to complete this paper.  The second author would like to thank BICMR for its hospitality while this paper was written.} of the geometric unit root L-functions.


\section{Introduction}

Let $\Fq$ be the finite field of $q$ elements with characteristic
$p$. Let $C$ be a smooth projective geometrically connected curve
defined over $\Fq$ with function field $K$. Let $U$ be a Zariski
open dense subset of $C$ with inclusion map $j: U\hookrightarrow C$. 
Let $G_K={\rm Gal}(K^{\rm sep}/K)$ denote the absolute Galois
group of $K$. For example, we can take $C={\mathbb P}^1$,
$U={\mathbb P}^1 -\{ 0, \infty\}$ and $K=\Fq(t)$.

Let $\pi_1^{\rm arith}(U)$ denote the arithmetic fundamental group
of $U$. That is,
$$\pi_1^{\rm arith}(U) = G_K/<I_x>_{x\in |U|},$$
where the denominator denotes the closed normal subgroup generated
by the inertial subgroups $I_x$ as $x$ runs over the closed points
$|U|$ of $U$. Let $D_x$ denote the decomposition group of $G_K$ at
$x$. One has the following exact sequence
$$1\rightarrow I_x \rightarrow D_x \rightarrow {\rm Gal}(\bar{k}_x/k_x) \rightarrow 1,$$ where $k_x$ denotes the
residue field of $K$ at $x$. The Galois group ${\rm Gal}(\bar{k}_x/k_x)$ is topologically generated by the geometric
Frobenius element ${\rm Frob}_x$ which is characterized by the property:
$${\rm Frob}_x^{-1}: \alpha \rightarrow \alpha^{\# k_x}.$$

Let $P_x$ denote the $p$-Sylow subgroup of $I_x$. Then we have the
following exact sequence
$$1\rightarrow P_x \rightarrow I_x \rightarrow I_x^{\rm tame}=\prod_{\ell\not=p}\Z_{\ell}(1) \rightarrow 1.$$

Let $F_{\ell}$ be a finite extension of ${\mathbb Q}_{\ell}$,
where $\ell$ is a prime number which may or may not equal to $p$.  Let $V$ be a finite dimensional vector
space over $F_{\ell}$. Let
$$\rho: G_K \longrightarrow GL(V)$$ be a continuous $\ell$-adic 
representation of $G_K$ unramified on $U$. Equivalently,
$$\rho: \pi_1^{\rm arith}(U)\longrightarrow GL(V)$$
is a continuous representation of $\pi_1^{\rm arith}(U)$.  The representation $\rho$ is called {\sl geometric} if it comes  
from an $\ell$-adic cohomology  of a smooth proper variety over $U$.  The geometric representations 
are the most interesting ones in applications. 

Given a representation $\rho$, its 
L-function is defined by
$$L(U, \rho, T) =\prod_{x\in |U|}\frac{1}{{\rm det}(I-\rho({\rm Frob}_x)T^{{\rm deg}(x)}|V)}\in 1+TR_{\ell}[[T]],$$
where $R_{\ell}$ is the ring of
integers in $F_{\ell}$. 
It is clear that this L-function is trivially $\ell$-adic analytic in the open unit disc $|T|_{\ell} <1$. 


We are interested in further analytic properties of this 
L-function $L(U, \rho, T)$, especially for those representations which come from
geometry \footnote{ Our definition of
a geometric representation depends on $\ell$. If $\ell\neq p$, then it forms a compatible system for all $\ell\neq p$ and one allows sub-quotients in the definition of geometric representations. In the case
$\ell=p$, our definition does not allow sub-quotients. In fact,
this was raised as an open problem in Remark 3.4.}.  More precisely, we want to know 

\begin{question} [Meromorphic continuation] When and where the L-function $L(U, \rho, T)$ is $\ell$-adic meromorphic? 
\end{question}

\begin{question}[Artin's conjecture]  
Assume that $\rho$ has no geometrically trivial component. When and where the L-function $L(U,\rho, T)$ is 
$\ell$-adic entire (no poles or analytic)? 
\end{question}

The answer depends very much on whether $\ell$ equals to $p$ or not. In the easier case $\ell \not =p$, the 
Grothendieck trace formula \cite{Gr} gives the following complete answer. 

\begin{thm} Assume that $\ell\not =p$. The L-function $L(U, \rho, T)$ is a rational function in $F_{\ell}(T)$. 
If $\rho$ has no geometrically trivial component, then $L(U, \rho, T)$ is a  
polynomial in $F_{\ell}[T]$. 

\end{thm}

In the case $\ell =p$, the situation is much more subtle. A general conjecture of Katz \cite{Ka} as proved by Emerton-Kisin \cite{EK}
says that the above two questions still have a complete positive answer if we restrict to the closed unit disc. 
That is, we have 

\begin{thm} Assume that $\ell=p$. The L-function $L(U, \rho, T)$ is $p$-adic meromorphic on the 
closed unit disc $|T|_p\leq 1$. If $\rho$ has no geometrically trivial component, then the L-function 
$L(U, \rho, T)$ is $p$-adic analytic (no poles) on  the closed unit disc $|T|_p\leq 1$. 
\end{thm}

The extension of the above results to larger $p$-adic disc is more subtle. For any given $\epsilon >0$, there 
are examples \cite{W1} showing that the L-function $L(U,\rho, T)$ is not $p$-adic meromorphic in the disc $|T|_p < 1 +\epsilon$, 
disproving another conjecture of Katz \cite{Ka}. However, if $\rho$ comes from geometry, then Dwork's conjecture \cite{Dw}
as proved by the second author \cite{W2}\cite{W3} shows the L-function is  indeed a good $p$-adic function: 

\begin{thm} Assume that $\ell=p$. If $\rho$ comes from geometry, then 
the L-function $L(U, \rho, T)$ is $p$-adic meromorphic in the whole 
$p$-adic plane $|T|_p < \infty$. 
\end{thm}

The aim of this paper is to study Artin's entireness conjecture for such L-functions of 
geometric $p$-adic representations. Our main result is the following theorem. 

 \begin{thm} Assume that $\ell=p$ and $\rho$ comes from geometry with no geometrically trivial components.   
 Then, there is a positive constant $c(p,\rho)$ such that 
the L-function $L(U, \rho, T)$  is $p$-adic analytic (no poles)  in the larger disc  
$|T|_p < 1+ c(p, \rho)$. Furthermore, there are geometrically non-trivial rank one geometric $p$-adic representations 
$\rho$ such that $L(U, \rho, T)$ is not $p$-adic analytic (in fact having infinitely many poles) in $|T|_p <\infty$. 
\end{thm}

The second part of the theorem shows that Artin's conjecture is false in the entire plane $|T|_p<\infty$. 
It shows that the first part of the theorem is best one can hope for, and Artin's conjecture is true 
in a larger disk than the closed unit disk for geometric $p$-adic representations. An interesting further question is how big the constant 
$c(p, \rho)$ can be. Our proof gives an explicit positive constant depending only on $p$ and 
some embedding rank of $\rho$. If $\rho$ comes from the slope zero part  of an ordinary overconvergent $F$-crystal on $U$ and 
the uniformizer of $R_p$ is $p$,  one can take $c(p, \rho) = p-1$ which is 
independent of $\rho$.


\section{$\ell$-adic case: $\ell\not=p$}

Since $\ell\not=p$, the restriction of the $\ell$-adic
representation $\rho$ to $P_x$ is of finite order and thus the
representation $\rho$ is almost tame. In fact, by class field
theory, $\rho$ itself has finite order up to a twist if $\rho$ has
rank one. Thus, there are not too many such $\ell$-adic
representations. The L-function $L(U, \rho, T)$ is always a rational
function. This follows from Grothendieck's trace formula
\cite{Gr}:

\begin{thm} Let ${\cal F}_{\rho}$ denote the lisse $\ell$-adic sheaf on
$U$ associated with $\rho$. Then, there are finite dimensional
vector spaces $H_c^i(U\otimes \bar{\mathbb{F}}_q, {\cal F}_{\rho})$
($i=0,1,2$) over $F_{\ell}$ such that
$$L(U, \rho, T) = \prod_{i=0}^2 {\rm det}(I-{\rm Frob}_q T|H_c^i(U\otimes \bar{\mathbb{F}}_q, {\cal F}_{\rho}))^{(-1)^{i-1}} \in F_{\ell}(T).$$
\end{thm}

If $U$ is affine, then $H_c^0=0$. If $\rho$ does not contain a
geometrically trivial component, then $H_c^2=0$. Thus, in most
cases, it is $H_c^1$ that is the most interesting.

\begin{cor}
Let $U$ be affine. Assume that $\rho$ does not contain a
geometrically trivial component. Then, the L-function
$$L(U, \rho, T) = {\rm det}(I-{\rm Frob}_q T|H_c^1(U\otimes \bar{\mathbb{F}}_q, {\cal F}_{\rho}))$$ is a polynomial.
\end{cor}

This is the $\ell$-adic function field analogue of Artin's
entireness conjecture.

Fix an embedding $\iota: \bar{\mathbb Q}_{\ell} \hookrightarrow {\mathbb C}$. A representation $\rho$ is called $\iota$-pure of
weight $w\in {\mathbb R}$ if each eigenvalue of ${\rm Frob}_x$
acting on $V$ has absolute value $q^{{\rm deg}(x)w/2}$ for all
$x\in |U|$. A representation $\rho$ is called $\iota$-mixed of
weights at most $w$ if each irreducible subquotient of $\rho$ is
$\iota$-pure of weights at most $w$. If $\rho$ is $\iota$-pure of
weight $w$ for every embedding $\iota$, then $\rho$ is called pure
of weight $w$. Similarly, if $\rho$ is $\iota$-mixed of weights at
most $w$ for every $\iota$, then $\rho$ is called mixed of weights
at most $w$. The fundamental theorem of Deligne \cite{De} on the
Weil conjectures implies

\begin{thm}
If $\rho$ is geometric, then $\rho$ is mixed with integral
weights. Furthermore, if $\rho$ is mixed of weights at most $w$,
then $H_c^i(U\otimes \bar{\mathbb{F}}_q, {\cal F}_{\rho})$ is mixed of
weights at most $w+i$.
\end{thm}

The $\ell$-adic function field Langlands conjecture for $\mathrm{GL}(n)$, which was established
by Lafforgue \cite{La}, implies

\begin{thm}
If $\rho$ is irreducible, then $\rho$ is geometric up to a twist
and hence pure of some weight.

\end{thm}
Thus, in the $\ell$-adic case with $\ell\not=p$, essentially all
$\ell$-adic representations are geometric from the viewpoint of L-functions.

\section{$p$-adic case}

In the case $\ell=p$, the restriction of the $p$-adic
representation $\rho$ to $P_x$ can be infinite and thus $\rho$ can
be very wildly ramified. The L-function $L(U, \rho, T)$ is naturally
more complicated and cannot be rational in general. One can ask
for its $p$-adic meromorphic continuation. The function $L(U, \rho, T)$ is trivially 
$p$-adic analytic in the open unit disc $|T|_p <1$ as the coefficients are in the ring $R_p$. It was shown in
\cite{W1} that $L(U, \rho, T)$ is not $p$-adic meromorphic in
general,  disproving a conjecture of Katz \cite{Ka}. However, one
can show that $L(U,\rho, T)$ is $p$-adic meromorphic on the closed
unit disc $|T|_p\leq 1$. Its zeros and poles on the closed unit
disc are controlled by $p$-adic \'etale cohomology of $\rho$. This
was proved by Emerton-Kisin \cite{EK}, confirming a conjecture of
Katz \cite{Ka}. That is,

\begin{thm}
For any $p$-adic representation $\rho$ of $\pi_1^{\rm arith}(U)$,
the quotient
$$\frac{L(U, \rho, T)}{\prod_{i=0}^2 {\rm det}(I-{\rm Frob}_q T|H_c^i(U\otimes \bar{\mathbb{F}}_q, 
{\cal F}_{\rho}))^{(-1)^{i-1}}}$$ has no zeros and poles on the closed
unit disc $|T|_p\leq 1$.
\end{thm}

In the case that $\rho$ has rank one, this was first proved by
Crew \cite{Cr}. Note that $H_c^2(U\otimes \bar{\mathbb{F}}_q, {\cal F}_{\rho})=0$ since $U$ is a curve and $\ell=p$.
If $U$ is affine, then $H_c^0(U\otimes \bar{\mathbb{F}}_q, {\cal F}_{\rho})=0$. This gives 

\begin{cor}
Let $U$ be affine. Then, the L-function
$L(U, \rho, T)$ is $p$-adic analytic on the closed unit disc $|T|_p\leq 1$. 
\end{cor}

The (compatible) $p$-adic analogue of a lisse $\ell$-adic sheaf (or $\ell$-adic 
representation) on $U$ for $\ell\not =p$ is  an overconvergent $F$-isocrystal over $U$, 
which is not a $p$-adic representation. Its pure slope parts, under the Newton-Hodge decomposition, 
are $p$-adic representations up to twists (unit root F-isocrystals, no longer overconvergent in general). 
$P$-adic representations arising in this way are also called \emph{geometric}, as they are a natural generalization 
of the geometric representations we defined before. 
For geometric $p$-adic representations, the following meromorphic
continuation was conjectured by Dwork \cite{Dw} and proved by the second author 
\cite{W2} \cite{W3}.

\begin{thm}
If the $p$-adic representation $\rho$ is geometric, then the
L-function $L(U, \rho, T)$ is $p$-adic meromorphic everywhere.
\end{thm}

\begin{rem}  It would be interesting to know 
if a sub-quotient of a geometric $p$-adic representation remains geometric in terms of our general definition. 
\end{rem}

Unlike the $\ell$-adic case, most $p$-adic representations are not
geometric. It seems very difficult to classify geometric $p$-adic
representations, even in the rank one case. This may be viewed as
the $p$-adic Langlands program for function fields of characteristic $p$, which is still wide
open\footnote{The $p$-adic Langlands correspondence for overconvergent $F$-isocrystals was recently studied by Tomoyuki Abe. 
This is compatible with the $\ell$-adic situation with $\ell\not =p$.  In our framework of convergent geometric 
unit root $F$-isocrystals, no one knows how to formulate the correspondence  yet,  even in the rank one case. In fact, 
our negative result below on the Artin entireness conjecture suggests something completely new happens in this new situation. }.

Our first new result of this paper is to show that the Artin entireness conjecture fails for
L-functions $L(U, \rho, T)$ of geometric $p$-adic representations,
even for non-trivial rank one $\rho$. 

\begin{thm} There are geometrically non-trivial rank one geometric $p$-adic representations $\rho$ on certain affine curves $U$ 
over $\mathbb{F}_p$ such that the L-function $L(U, \rho, T)$ is $p$-adic meromorphic on $|T|_p<\infty$, but having 
infinitely many poles. 
\end{thm}

{\bf Proof}. Let $p>2$ be an odd prime and $N>4$ be a positive integer prime to $p$. 
Let $Y$ be the component of ordinary non-cuspidal locus of the modulo $p$ reduction of the compactified modular curve $X_1(Np)$. 
This is an affine curve over the finite field $\mathbb{F}_p$. Let $E_1(Np)$ be the universal elliptic curve over $Y$. 
Its relative $p$-adic \'etale cohomology is a rank one geometric $p$-adic representation $\rho$ of $\pi_1^{\rm arith}(Y)$. 
For a non-zero integer $k$, the $k$-th tensor power $\rho^{\otimes k}$ is 
again a geometric $p$-adic representation of $\pi_1^{\rm arith}(Y)$. 
The Monsky trace formula 
 gives the following relation
 \begin{equation}
 \label{E: unitL}
L(Y, \rho^{\otimes k}, T) =  \frac{D(k+2, T)}{D(k, pT)} , 
\end{equation}
where $D(k,T)$ is the characteristic power series of the $U_p$-operator acting on the space of overconvergent 
$p$-adic cusp forms of weight $k$ and tame level $N$. The series $D(k, T)$ is a $p$-adic entire function.  
Equation \eqref{E: unitL} implies that the L-function $L(Y, \rho^{\otimes k}, T)$ is $p$-adic meromorphic 
in $T$, which was first proved by Dwork in \cite{Dw71} via Monsky's trace formula, see also \cite{Ka73} 
and \cite{Co}. 

We want to show that the L-function $L(Y, \rho^{\otimes k}, T)$ is not $p$-adic entire 
for infinitely many integers $k$. For this purpose, we need to describe the coefficients of the L-function 
in more detail, following Coleman \cite[Appendix I]{Co}.  In the following, we fix an embedding $\iota_p:\overline{\mathbb{Q}}\hookrightarrow\overline{\mathbb{Q}}_p$.

For an order $\cal{O}$ in a number field, let $h(\cal{O})$ denote the class number of $\cal{O}$. 
If $\gamma$ is an algebraic integer, let $O_{\gamma}$ be the set of orders in $\mathbb{Q}(\gamma)$ containing 
$\gamma$. For a positive integer $m$, let $W_{p,m}$ denote the finite set of $p$-adic units  
$\gamma\in \mathbb{Q}_p$ such that $\mathbb{Q}(\gamma)$ is an imaginary quadratic field,  
$\gamma$ is an algebraic integer 
and
$${\text{Norm}}^{\mathbb{Q}(\gamma)}_{\mathbb{Q}}(\gamma)=p^m.$$
By Coleman \cite[Theorem I1]{Co}, for all integers $k$, we have 
$$D(k, T) = \exp( \sum_{m=1}^{\infty} A_m(k) \frac{T^m}{m}),$$
where 
\[
A_m(k) = \sum_{\gamma \in W_{p,m}} \sum_{\mathcal{O} \in O_{\gamma}} h({\cal{O}})
B_N({\cal{O}}, \gamma) \frac{\gamma^k}{\gamma^2 -p^m},
\]
and $B_N(\cal{O}, \gamma)$ is the number of elements of ${\cal{O}}/N{\cal{O}}$ of 
order $N$ fixed under multiplication by $p^m/\gamma$. This is really another form of the Monsky trace formula. 
It follows that 
\[
L(Y, \rho^{\otimes k}, T) =\exp(\sum_{m=1}^{\infty} C_m(k)\frac{T^m}{m}), 
\]
where 
\[
C_m(k) = A_m(k+2) - A_m(k)p^m  =  \sum_{\gamma \in W_{p,m}} \sum_{\mathcal{O} \in O_{\gamma}} h({\cal{O}})
B_N({\cal{O}}, \gamma) \gamma^k.
\] 
It is clear that $C_m(k)$ is an algebraic number 
in $\mathbb{Q}_p$. To proceed, we first recall the following basic fact regarding linear independence of square roots of integers. 

\begin{lemma}\label{L:linear-independence}
Let $1\leq n_1<n_2<\dots<n_l$ be square free integers. Suppose $a_1,\dots, a_l\in\mathbb{Q}$. If $a_i\neq 0$ for some $i$,
then $a_1\sqrt{n_1}+\dots+a_l\sqrt{n_l}\neq0$.
\end{lemma}
\begin{proof}
Suppose $p_1,\dots,p_s$ are the prime factors of all $n_i$. Then we may write 
\[
a_1\sqrt{n_1}+\dots+a_l\sqrt{n_l}=f(\sqrt{p_1},\dots,\sqrt{p_s}),
\]
where $f(y_1,\dots, y_s)$ is a polynomial with rational coefficients of degree at most 1 with respect to every $y_j$. Now the lemma follows from the main result of \cite{Be}.
\end{proof}

\begin{cor}
Keep notations as in Lemma \ref{L:linear-independence}. Suppose none of $a_1,\dots, a_l$ are $0$. Then 
\[
\mathbb{Q}(\sqrt{n_1},\dots,\sqrt{n_l})=\mathbb{Q}(a_1\sqrt{n_1}+\dots+a_l\sqrt{n_l}).
\]
\end{cor}
\begin{proof}
Put $c=a_1\sqrt{n_1}+\dots+a_l\sqrt{n_l}$. We need to show that $\sqrt{n_i}\in \mathbb{Q}(c)$ for every $i$. If not, without loss of generality, we may suppose $\sqrt{n_1}\notin \mathbb{Q}(c)$. Then there exists an automorphism $\sigma\in\Gal(\bar{\mathbb{Q}}/\mathbb{Q})$ such that $\sigma(c)=c$ and $\sigma(\sqrt{n_1})=-\sqrt{n_1}$. Putting these together, we will obtain a non-trivial linear relation of $\sqrt{n_1},\dots,\sqrt{n_l}$ over $\mathbb{Q}$; this contradicts with the above lemma. 
\end{proof}

We need the following key property.

\begin{lemma} For $k\geq1$, the field generated by all the algebraic numbers $C_m(k)$ in $\mathbb{Q}_p$ is equal to the compositum 
of all imaginary quadratic fields in $\mathbb{Q}_p$ in which $p$ splits. In particular, this 
field is an infinite algebraic extension of $\mathbb{Q}$ in $\mathbb{Q}_p$. 
 \end{lemma}

{\bf Proof}. Let $\gamma$ be a $p$-adic unit such that $\mathbb{Q}(\gamma)$ is an imaginary quadratic field and ${\text{Norm}}^{\mathbb{Q}(\gamma)}_{\mathbb{Q}}(\gamma)=p^m$. It is clear that in this case $p$ splits in $\mathbb{Q}(\gamma)$. Thus $C_m(k)$ is contained in the compositum 
of all imaginary quadratic fields in $\mathbb{Q}_p$ in which $p$ splits.  

Conversely, let $K$ be any imaginary quadratic field in $\mathbb{Q}_p$ in which $p$ splits. Write 
$p\mathcal{O}_K=\mathfrak{p}\bar{\mathfrak{p}}$. Without loss of generality, we may suppose $\bar{\mathfrak{p}}=p\mathbb{Z}_p\cap\mathcal{O}_K$. Write $m=h(\mathcal{O}_K)$, then $\mathfrak{p}^m=(\gamma)$ is a principal ideal, and $\bar{\mathfrak{p}}^m=(\bar \gamma)$. It  follows that ${\text{Norm}}^{\mathbb{Q}(\gamma)}_{\mathbb{Q}}(\gamma)=\gamma\bar \gamma=p^m$ and $\gamma$ is a $p$-adic unit. By replacing $\gamma$ with $\gamma^n$ and $m$ with $mn$ for some suitable positive integer $n$, we may further suppose that
$\bar \gamma \equiv 1\mod N\mathcal{O}_K$. 
In particular, we have $B_N(\mathcal{O},\gamma)>0$ for any $\mathcal{O}\in O_\gamma$. 

We claim that for $\gamma'\in K\cap W_{p,m}$, if $B_N(\mathcal{O},\gamma')>0$ for some $\mathcal{O}\in O_{\gamma'}$, then $\gamma'=\gamma$.   In fact, since
${\text{Norm}}^{K}_{\mathbb{Q}}(\gamma')=p^m$, we first have that 
$(\gamma')=\mathfrak{p}^i\bar{\mathfrak{p}}^j$ with $i+j=m$. Note thtat $\gamma'$ is a $p$-adic unit; this yields $j=0$. That is $(\gamma)=(\gamma')$. Thus we may write $\gamma'=u\gamma$ for some unit $u$. 

Now choose $a\in \mathcal{O}$ such that its image in $\mathcal{O}/N\mathcal{O}$ is of order $N$ and fixed by $\bar{\gamma'}$. Since $\bar \gamma \equiv 1\mod N\mathcal{O}_K$, we get that the product $(\bar{u}-1)a$ is zero in the quotient group $\mathcal{O}/N\mathcal{O}$.   From
the assumption that a has order exactly $N$, one deduces that for every
prime $l$ dividing $N$, the $l$-adic valuation $v_{l}(\bar{u}-1) \geq
v_{l}(N)$. It follows that if $u$ is not 1, then $(u-1)(\bar{u}-1) \geq N$.
Since $u\bar{u}=1$, this gives the inequality $4 < N \leq (u-1)(\bar{u}-1)
\leq 4$ as $u$  has complex norm 1,  which is a contradiction if $u$ is not $1$!
 
Consequently, we may write
\[
C_m(k)=(\sum_{\mathcal{O}\in O_{\gamma}}h(\mathcal{O})B_N(\mathcal{O},\gamma))\gamma^k+\alpha
\] 
where $\alpha$ is a sum of  elements contained in quadratic fields different from $K$.  Since  $\gamma^k$ and 
$\bar{\gamma}^k$ has different $p$-adic valuation, we deduce that $\gamma^k\notin\mathbb{Q}$. Thus $\mathbb{Q}(\gamma)=\mathbb{Q}(\gamma^k)$. By the above corollary, we therefore conclude that 
$K=\mathbb{Q}(\gamma^k)$ is contained in the field generated by $C_m(k)$. This yields the lemma.

\bigskip

We now return to the proof of the theorem. Let $k\geq 1$ be a positive integer. 
Let ${\cal{F}}$ denote the relative rigid cohomology of $E_1(Np)$ over $Y$, which is 
an ordinary overconvergent $F$-isocrystal over $Y$ of rank two, sef-dual and pure of weight $1$. 
The rank one 
$p$-adic representation $\rho$ is precisely the unit root part of ${\cal F}$.  
It follows that the $L$-function of the $k$-th  Adams operation of ${\cal{F}}$ is 
$$L(Y, \rho^{\otimes k}, T) L(Y, \rho^{\otimes (-k)}, p^kT)
=  \frac{L(Y, \Sym^k{\cal F}, T)}{L(Y, \Sym^{k-2}{\cal F}, pT)}.$$
The right side is a rational function with integer coefficients. 
If both $L(Y, \rho^{\otimes k}, T)$ and $ L(Y, \rho^{\otimes (-k)}, T)$ had a finite number of poles, 
then the  above left side would be a $p$-adic meromorphic function with a finite number of poles, and it is at the 
same time a rational function. 
It would then follow that both  $L(Y, \rho^{\otimes k}, T)$ and $ L(Y, \rho^{\otimes (-k)}, T)$ have a finite number of zeros and thus both 
would be rational functions. This implies that the coefficients of $L(Y, \rho^{\otimes k}, T)$ and 
$ L(Y, \rho^{\otimes (-k)}, T)$ generate a finite algebraic extension of $\mathbb{Q}$ in $\mathbb{Q}_p$, 
contradicting to the lemma. We conclude that at least one of the two functions 
$L(Y, \rho^{\otimes k}, T)$ and $ L(Y, \rho^{\otimes (-k)}, T)$ has  infinitely many poles. 
The theorem is proved.

\begin{rem}
For any positive integer $k\geq 1$, we believe that both functions $L(Y, \rho^{\otimes k}, T)$ and $ L(Y, \rho^{\otimes (-k)}, T)$ 
have  infinitely many poles. But we do not know how to prove it. 
\end{rem}

\begin{rem}
  In the analogous setting of the family of Kloosterman sums, the
unit root L-function is again expected to be non-rational, but this
remains unknown at present, see page 4 in \cite{Ha}.
\end{rem}
Our second result of this paper is to show that for a geometric $p$-adic representation $\rho$ on a smooth affine 
curve $U$ over $\mathbb{F}_p$, the L-function $L(U, \rho, T)$ is $p$-adic analytic (no poles)  in the larger disc  
$|T|_p < 1+ c(p, \rho)$ for some positive constant $c(p, \rho)$. In fact, we shall prove a more general 
theorem in the context of $\sigma$-modules as in \cite{W2}\cite{W3}. For simplicity of notation, we use $L(\rho, T)$ 
to denote $L(U, \rho, T)$.

 \begin{thm} Let $U$ be a smooth affine curve over $\mathbb{F}_q$. Let $\rho$ be a unit root $\sigma$-module 
 which arises as a pure slope part of an overconvergent $\sigma$-module on $U$. Then, 
 there is a positive constant $c(p,\rho)$ such that 
the L-function $L(\rho, T)$  is $p$-adic analytic (no poles)  in the larger disc  
$|T|_p < 1+ c(p, \rho)$. 
\end{thm}

\begin{proof} Let $\phi$ be an overconvergent $\sigma$-module 
on $U$ with coefficients in $R_p$ with uniformizer $\pi$. Since $\phi$ is overconvergent, 
Corollary 3.2 in \cite{W3} shows that its L-function $L(\phi, T)$ is $p$-adic meromorphic 
everywhere. As $U$ is a smooth affine curve, Corollary 3.3 in \cite{W3} further shows that 
$L(\phi, T)$ is $p$-adic analytic in the disk  $|T|_p < |\pi^{-1}|_p$. Note that $|\pi^{-1}|_p$ 
is a constant greater than $1$. For example, in the case $\pi=p$, we have $|\pi^{-1}|_p=p$.

We first assume that  $\phi$ is ordinary. 
For an integer $i\geq 0$, let $\phi_i$ denote the unit root $\sigma$-module on $U$ coming from 
the slope $i$-part in the Hodge-Newton decomposition of $\phi$. 
It is no longer overconvergent in general. 
We need to show that the unit root $\sigma$-module 
L-function $L(\phi_i, T)$ is $p$-adic analytic in the disk $|T|_p < |\pi^{-1}|_p$. 
By the definition of $\phi_i$ and our ordinariness assumption, 
we have the decomposition 
$$L(\phi, T) =\prod_{i\geq 0} L(\phi_i, \pi^i T)=L(\phi_0, T) \prod_{i\geq 1}L(\phi_i, \pi^iT).$$
As mentioned above, the left side is $p$-adic analytic in the disk  $|T|_p < |\pi^{-1}|_p$. For each $i\geq 1$, the right side factor  
$L(\phi_i, \pi^iT)$ is trivially $p$-adic analytic with no zeros and poles in the disk  $|T|_p < |\pi^{-1}|_p$. 
We deduce that the first right side factor $L(\phi_0, T)$ is also $p$-adic analytic in the disk  $|T|_p < |\pi^{-1}|_p$. 
This proves the theorem in the case $i=0$. 

For $i>0$, we need to use the proof of Dwork's conjecture in \cite{W2}\cite{W3}. 
Let $\psi =\phi_i$. We need to prove that $L(\psi, T)$ is $p$-adic analytic in the disk  $|T|_p < |\pi^{-1}|_p$. 
Let $r_i$ denote the rank of $\phi_i$.
Define 
$$\tau= \wedge^{r_0}\phi_0 \otimes \wedge^{r_1}\phi_1 \otimes \cdots \otimes \wedge^{r_{i-1}}\phi_{i-1}.$$ 
This is a rank one unit root $\sigma$-module on $U$, not overconvergent in general.  
Define 
$$\varphi = \pi^{ -r_1-\cdots -(i-1)r_{i-1} -i} \wedge^{r_0+\cdots +r_{i-1}+1}\phi.$$
Since $\phi$ is ordinary and overconvergent, it follows that $\varphi$ is also 
ordinary and overconvergent. For an integer $j\geq 0$, let $\varphi_j$ denote the unit root $\sigma$-module on $U$ coming from 
the slope $j$-part in the Hodge-Newton decomposition of $\varphi$. 
Then, it is easy to check that we have the following decomposition (see equation (5.1) in \cite{W3}). 
$$L(\varphi \otimes \tau^{-1}, T) =L(\psi, T) \prod_{j\geq 1} L(\varphi_j \otimes \tau^{-1}, \pi^j T).$$ 
For each $j\geq 1$, the factor $L(\varphi_j\otimes \tau^{-1}, \pi^jT)$ is 
trivially $p$-adic analytic with no zeros and poles in the disk  $|T|_p < |\pi^{-1}|_p$. 
To prove that $L(\psi, T)$ is also $p$-adic analytic in the disk  $|T|_p < |\pi^{-1}|_p$, 
it suffices to prove that the left side factor $L(\varphi \otimes \tau^{-1}, T)$ is $p$-adic analytic in the disk  $|T|_p < |\pi^{-1}|_p$. 

Now, the rank one unit root $\sigma$-module $\tau$ is the slope zero part of the 
following ordinary and overconvergent $\sigma$-module
$$\Phi = \pi^{ -r_1-\cdots -(i-1)r_{i-1}} \wedge^{r_0+\cdots +r_{i-1}}\phi.$$
By Theorem 7.8 in \cite{W3}, we deduce that there is a sequence of nuclear overconvergent 
$\sigma$-modules $\Phi_{\infty, -k}$ ($k\geq 2$) such that 
$$L(\varphi \otimes \tau^{-1}, T) =\prod_{k\geq 1} L(\varphi \otimes \Phi_{\infty, -1-k} \otimes \wedge^k \Phi, T)^{(-1)^{k-1}k}.$$
Since $\Phi$ is ordinary and its slope zero part has rank one, $\wedge^k\Phi$ is divisible by 
$\pi^{k-1}$. It follows that for $k\geq 2$, the L-function $L(\varphi \otimes \Phi_{\infty, -1-k} \otimes \wedge^k \Phi, T)$ 
is trivially $p$-adic analytic with no zeros and poles in the disk  $|T|_p < |\pi^{-1}|_p$. 
For the remaining case $k=1$, we apply the one dimensional case of the following $n$-dimensional integrality result   
and deduce that $L(\varphi \otimes \Phi_{\infty, -2} \otimes \Phi, T)$ is 
$p$-adic analytic in the disk  $|T|_p < |\pi^{-1}|_p$. The theorem is proved in the ordinary case. 

\begin{lem} Let $U$ be a smooth affine variety of equi-dimension $n$ over $\mathbb{F}_q$. Let $\phi$ be an overconvergent nuclear 
	$\sigma$-module on $U$. Then, 
	the L-function $L(\phi, T)^{(-1)^{n-1}}$  is $p$-adic analytic (no poles)  in the disc  
	$|T|_p < |\pi^{-1}|_p$. 
\end{lem}

\begin{proof}
	In the case that $\phi$ has finite rank, this integrality is already proved in Corollary 3.3 in \cite{W3}. 
	In this case,  the finite rank Monsky trace formula (Theorem 3.1 in \cite{W3}) 
states that
$$L(\phi, T)^{(-1)^{n-1}} = \prod_{i=0}^n \det(I -\phi_i^*T|M_i^*\otimes_RK)^{(-1)^i},$$
where $R=R_p$ in our current notation and $\det(I -\phi_i^*T|M_i^*\otimes_R K)\in  1+TR[[T]]$ is a $p$-adic entire function. Now, the divisibility 
$\phi_i^* \equiv 0~(\mathrm{mod}~\pi^i)$ (equation (3.4) in \cite{W3}) shows that 
$$\det(I -\phi_i^*T|M_i^*\otimes_RK) \in 1+\pi^iTR[[\pi^iT]].$$ 
This implies the integrality in the finite rank case. 
	For infinite rank nuclear overconvergent $\sigma$-modules, the proof is the same. One simply uses the infinite rank 
	nuclear overconvergent trace formula (Theorem 5.8 in \cite{W2}): 
	$$L(\phi, T)^{(-1)^{n-1}} = \prod_{i=0}^n \det(I -\Theta_iT|M_i^*\otimes_RK)^{(-1)^i}.$$
	Note that there is a misprint of indices in Theorem 5.8 in \cite{W2}: $\det(I -\Theta_iT|M_i^*\otimes_RK)$ there should be 
	$ \det(I -\Theta_{n-i}T|M_{n-i}^*\otimes_RK)$, compare the finite rank case (Theorem 3.1 in \cite{W3}). 
	Now, one uses the same divisibility $\Theta_i \equiv 0 ~(\mathrm{mod}~\pi^i)$, which follows from the definition of 
	$\Theta_i$ (Definition 5.5 in \cite{W2}) and the fact that $\phi_i =\phi\otimes \sigma_i$ is divisible by $\pi^i$.

\end{proof}

In the general non-ordinary case, by a similar argument, we may apply
the methods in \cite{W2}\cite{W3} for non-ordinary case   
to give an explicit positive constant $c(p, \rho)$ depending on $\pi$ and the rank of $\phi$ 
such that $L(\rho, T)$ is $p$-adic analytic in the disk  $|T|_p < 1 + c(p, \rho)$. 
The constant $c(p, \rho)$ depends very badly on the rank of $\phi$, and 
so we would not bother to write it down explicitly. 
\end{proof}

The above theorem has a higher dimensional generalization. 
We state this generalization below.

\begin{rem} Let $U$ be a smooth affine variety of equi-dimension $n$  over $\mathbb{F}_q$. Let $\rho$ be a unit root $\sigma$-module on $U$. 
Then, the L-function $L(\rho, T)^{(-1)^{n-1}}$ is $p$-adic analytic on the closed unit disk $|T|_p \leq 1$.   
If $\rho$ arises as a pure slope part of an overconvergent $\sigma$-module on $U$. Then, 
 there is a positive constant $c(p,\rho)$ such that 
the L-function $L(\rho, T)^{(-1)^{n-1}}$  is $p$-adic analytic (no poles)  in the larger disc  
$|T|_p < 1+ c(p, \rho)$. 
\end{rem}

The first part follows from Emerton-Kisin's theorem on the Katz conjecture and standard properties of $p$-adic \'etale cohomology.  
The proof of the second part is the same as the above theorem, and use the results in \cite{W2}\cite{W3}. We expect that both parts  
remain true if $U$ is an equi-dimensional complete intersection (possibly singular) in a smooth affine variety $X$ over $\mathbb{F}_q$. 
This possible generalization is motivated by the characteristic $p$ entireness result in \cite{TW}. 

\begin{rem} 
In the special case that $U$ is the compliment of a hypersurface
and the Frobenius lifting is the $q$-th power lifting of the
coordinates, the result of Dwork-Sperber \cite{DS} can be used
to prove the $p$-adic analytic continuation of $L(\rho, T)^{(-1)^{n-1}}$
for geometric ordinary $\rho$ in the open disc
$|T|_p<p^{(p-1)/(p+1)}$. This result is weaker than our result since
the disc $|T|_p<p^{(p-1)/(p+1)}$ is smaller than the disc
$|T|_p<p$ obtained in our approach.
\end{rem}

\end{document}